\documentclass[11pt, english]{amsart}
\usepackage{amssymb,amsthm,amsmath,amstext}
\usepackage{mathrsfs}
\usepackage{bm}
\usepackage[utf8]{inputenc}

\usepackage[all]{xy}

\makeatletter
\ifnum\@ptsize=0  \fi \ifnum\@ptsize=2  \fi 

\title{Proof of a conjecture of Sturmfels, \linebreak Timme and Zwiernik}
\author{Laurent Manivel}
\date{\today}

\theoremstyle{plain}
\newtheorem{theorem}{Theorem}
\newtheorem*{conjecture*}{Conjecture}
\newtheorem{definition}[theorem]{Definition}
\newtheorem{prop}[theorem]{Proposition}
\newtheorem{lemma}[theorem]{Lemma}

\def\CC{{\mathbb{C}}}

\def\PP{{\mathbb{P}}}
\def\QQ{{\mathbb{Q}}}
\def\SS{{\mathbb{S}}}

\def\cC{{\mathcal{C}}}

\def\cL{{\mathcal{L}}}
\def\cM{{\mathcal{M}}}
\def\cQ{{\mathcal{Q}}}
\def\cU{{\mathcal{U}}}

\begin{document}

\begin{abstract}
We prove a conjecture of Sturmfels, Timme and Zwiernik on the ML-degrees of linear covariance models in 
algebraic statistics. As in our previous works on 
linear concentration models, the proof ultimately relies 
on the computation of certain intersection numbers on the varieties of complete quadrics. 
\end{abstract}

\maketitle

 \section{Introduction}
In algebraic statistics, there is a notion of maximum likelihood estimation whose complexity is governed by a fundamental
invariant called maximum likelihood degree or {\it ML-degree} (see e.g. \cite{icm, uhler, huh} for an introduction). This 
degree depends very much on the statistical model; here we consider the so-called linear covariance models, whose
ML-degree defines a rather mysterious invariant of a space of symmetric matrices. These models are very different from 
the linear concentration models considered in \cite{mmmsv}, for which we obtained quite explicit polynomial formulas for the generic ML-degrees, proving conjectures of Sturmfels-Uhler and Nie-Ranestad-Sturmfels.  In this note we use closely 
related techniques to compute, in small dimension, the generic ML-degree of a linear covariance model. This degree 
can be defined very explicitely as 
follows \cite[Proposition 3.1]{stz}.
%\section{ML degree of a linear covariance model} 

\smallskip
Let us denote by $\SS_n$ the vector space of complex symmetric matrices of size $n$. It is endowed with the standard scalar product 
$\langle A, B\rangle= Tr(AB)$. 

\begin{definition}
Consider a subspace $\cL\subset \SS_n$, of dimension $m$. The ML-degree of $\cL$ is the number of solutions,
for $S$ a generic symmetric matrix, of the following system of linear and quadratic equations in $K$ and $\Sigma$:
\begin{equation}\label{ML-degree}\Sigma\in\cL, \quad   K\Sigma =I_n, \quad   KSK-K\in\cL^\perp.\end{equation}
\end{definition} 

We will compute the ML-degree $ML_m$ for $\cL$ {\it generic} of dimension 
$m$ up to four. 

\begin{theorem}
For $m\le 4$, the ML-degree $ML_m$ of a generic linear concentration model of dimension $m$ is:
\begin{eqnarray}
ML_2 &= & 2n-3, \\
ML_3 &= & 3n^2-9n+ 7, \\
ML_4 &= & \frac{11}{3}n^3-18n^2+ \frac{85}{3}n-15.
\end{eqnarray}
\end{theorem}

Statement (1) was proved in \cite{cmr}. Statements (2) and (3) were conjectured in \cite[Conjecture 4.2]{stz}. 

\smallskip
As in \cite{mmmsv} we will reduce the problem of computing $ML_m$ to a computation on the space of complete quadrics. There are two main difficulties. 
First, one has to understand the contributions of the exceptional divisors, dominating the loci of symmetric matrices 
of a given rank; up to dimension $m=4$, only corank one and two need to be taken into account, which strongly 
simplifies the problem. Second, we need to make an intersection product in a situation which is not quite generic, 
and we need to be cautious about the transversality conditions that are required to make the computation meaningful. 
Let us start with a brief reminder on complete quadrics. 

\section{Complete quadrics} 
The space of complete quadrics $CQ_n$ is a  compactification of the space of invertible symmetric matrices 
(up to scalar) $\PP(\SS_n)^\circ \subset \PP(\SS_n)$, or equivalently, of smooth quadrics in $\PP^{n-1}$, which 
is well-suited for enumerative geometry. Denote by $D_i\subset \PP(\SS_n)$ the degeneracy locus consisting 
of matrices of rank at most $i$, 
an irreducible variety whose singular locus is $D_{i-1}$. Note that by mapping  a rank $i$ matrix to its image, 
we get a morphism from $D_i-D_{i-1}$ to the Grassmannian $Gr(i,n)$, which is an open subset of the vector bundle
$S^2\cU$ for $\cU$ the tautological vector bundle of rank $i$. We will also denote by $\cQ$ the tautological quotient 
bundle of rank $n-i$ on  $Gr(i,n)$. 

\begin{definition}\label{def:CQ2}
The space of complete quadrics $CQ_n$ is the successive blow-up of $\PP(\SS_n)$ along the degeneracy loci:
$$
CQ_n = Bl_{\widetilde D_{n-2}} \ldots Bl_{\widetilde D_{2}}Bl_{D_1} \PP(\SS_n),
$$
where $\widetilde D_{i} $ is the (smooth) proper transform of $D_i$
by the previous blow-ups.
\end{definition}

See \cite{mmmsv} and references therein for more details and other equivalent definitions. We will denote by 
$N=\binom{n+1}{2}-1$ the dimension of $CQ_n$ and of $\PP(\SS_n)$. 

Note that by construction, $CQ_n$ admits a natural basis (over $\QQ$) of divisors $E_1,\ldots,E_{n-1}$ consisting of 
the strict transforms of the exceptional divisors of the successive blow-ups, plus $E_{n-1}$ which is the 
proper transform of the determinant hypersurface of singular matrices. These divisors are smooth and meet 
transversally in the variety of complete quadrics. They give to $CQ_n$ a kind of Russian doll structure since one can show that 
$$E_i\simeq CQ_i(\cU)\times_{G(i,n)} CQ_{n-i}(\cQ^\vee).$$
One of the most useful 
properties of $C_n$ is that is factorizes the inversion morphism $\iota$ for symmetric matrices:
there is a commutative diagram 
\begin{equation*}
\xymatrix{
 & CQ_n \ar[ld]_{p}\ar[rd]^{q} & \\ 
\PP(\SS_n) \ar@{.>}[rr]^{\iota} & & \PP(\SS_n),}
 \end{equation*}
 where $p$ is the cascade of blow-ups that defines $CQ_n$, while $q$ is obtained by contracting the exceptional
 divisors in reverse order, first  $E_2$ to $D_{n-2}$, 
 and so on up to $E_{n-1}$ contracting to $D_1$ (note that the comatrix of a matrix of corank one is a rank one matrix). 
 In particular $p$ and $q$ identify to $\PP(\SS_n)^\circ$ the complement $CQ_n^\circ$ of the union of the 
 exceptional divisors in $CQ_n$. 
 Of course this diagram is compatible with the relative product structure of each of the exceptional divisors $E_i$,
 that admit a natural contraction map to 
 $$F_i\simeq \PP(S^2\cU)\times_{G(i,n)} \PP(S^2\cQ^\vee).$$
 
 Finally, we will need to know that the pull-backs $H_1$ and $H_{n-1}$ of the hyperplane divisors on $\PP(\SS_n)$ 
 by $p$ and $q$ are  given in 
 terms of the exceptional divisors by the formulas
 $$nH_1 = (n-1)E_1 +(n-2)E_2 + \cdots + E_{n-1},$$
 $$nH_{n-1} = E_1 +2E_2 + \cdots + (n-1)E_{n-1}.$$
 In fact they are part of an alternative $\QQ$-basis $(H_1,H_2,
 \ldots , H_{n-1})$ of the Picard group of $CQ_n$, where $H_k$ is obtained by pulling-back the hyperplane divisor by the morphism defined by the $k\times k$ minors of the generic matrix (or the $(n-k)\times (n-k)$ minors of the inverse matrix). 
 
Now we will try to reduce the computation of the generic ML-degree to an enumerative problem on the variety of complete quadrics. 

\section{Degeneracy locus interpretation}
Recall that our main goal is to count the number of solutions of the system of equations (\ref{ML-degree}). 
We can replace  $K$ by $KS$ and $\Sigma$ by $S^{-1}\Sigma$ , and get rid of $S$ by noticing that $\cL^\perp S=(S^{-1}\cL)^\perp$. So our problem is  equivalent to counting the number of solutions, for $\cL$ generic, of the system 
\begin{equation}\label{ML-degree-bis}
\Sigma\in \cL, \quad   K\Sigma =I_n, \quad   K^2-K\in\cL^\perp.
\end{equation}
After homogeneizing, the last condition simply means that we want $[K]\in\PP(\SS_n)^\circ$ to be such that 
$K^2$ is proportional to $K$ modulo $\cM:=\cL^\perp$. This makes sense on $CQ_n$ and we would be tempted to try 
to count the number of points  $Q\in CQ_n$ verifying the previous conditions for $[K]=p(Q)$ and $[\Sigma]=q(Q)$. 
The proportionality condition may be interpreted as a degeneracy condition for the morphism between vector 
bundles on $CQ_n$ given by 
$$\phi_\cM: \mathcal{O}_{CQ_n}(-L_1)\oplus \mathcal{O}_{CQ_n}(-2L_1) \stackrel{(K,K^2)}{\longrightarrow} 
\mathcal{O}_{CQ_n}\otimes \SS_n/\mathcal{M}.$$
Let $D(\phi_\cM)$ denote the degeneracy locus in $CQ_n$, where this morphism is not injective. By definition 
$$ML_m = \# \Big( D(\phi_\cM)\cap q^{-1}\PP(\cL)^\circ\Big),$$
where the exponent means that we restrict to invertible matrices in $\cL$.

Recall that the expected codimension of the degeneracy locus of a morphism between vector bundles is the difference
between their ranks plus one. Here this gives $m-2+1=m-1$, which is the dimension of $\PP(\cL)$. So we expect,
if everything goes fine, 
a finite number of intersection points. Moreover these points  should be smooth points, which means that the intersection
of $D(\phi_\cM)$ and $q^{-1}\PP(\cL)$ should be transverse at every such point. This would necessarily be the case,
by general arguments, if we could replace $\PP(\cL)$ by some $\PP(\cL')$ (of the same dimension) that 
could be chosen generically. But here $\cL$ is directly involved in the definition of $\phi_\cM$, 
so we have to be careful. We prove the following statement.

\begin{lemma}
For $\cL$ generic, the intersection of  $D(\phi_\cM)$ with $q^{-1}\PP(\cL)^\circ$ is 
everywhere transverse. 
\end{lemma}

\begin{proof}
By the general irreducibility argument, it suffices to exhibit one transverse intersection point,
for some $\cL$. We will suppose that the identity matrix $I_n$  belongs to $\cL$ and that our intersection point is given by 
$K_0=I_n$. Then $K=I_n+J$ is tangent to the 
locus where $K^2$ is proportional to $K$ modulo $\cL^\perp$ iff $J$ belong to $\cL^\perp $. Since the inverse of 
$K$ is $I_n-J$ at first order, the 
transversality condition will be verified as soon as $\cL\cap \cL^\perp=0$. This means that the standard 
scalar product on $\SS_n$ must remain non degenerate when we restrict it to $\cL$, which is clearly the case in general. 
\end{proof}

Note that this immediately implies that the intersection of  $D(\phi_\cM)$ with $CQ_n^\circ$ has the 
expected codimension. Indeed, if its dimension was bigger than expected, it would have to intersect in 
$\PP(\SS_n)$ any linear space of dimension $m-1$ along a positive dimensional subvariety, and the
transversality property could not hold. 

\medskip
This being established, 
if we could prove that $D(\phi_\cM)$ and $q^{-1}\PP(\cL)$ intersect only in $CQ_n^\circ\simeq\PP(\SS_n)^\circ$, 
we would conclude that $ML_m$ can be computed from the intersection theory of $CQ_n$. Indeed, if $D(\phi_\cM)$ is 
of the expected codimension $m-1$, its fundamental class in the Chow ring of $CQ_n$ is given by the 
Thom-Porteous formula \cite[Chapter 14]{fulton}:
$$[D(\phi_\cM)]=c_{m-1}(\mathcal{O}_{CQ_n}\otimes \SS_n/\mathcal{M}-(\mathcal{O}_{CQ_n}(-H_1)\oplus \mathcal{O}_{CQ_n}(-2H_1))).$$
Here we need to recall that the Chern class of a formal difference $E-F$ between two vector bundles $E$ and $F$  
 is simply $c(E-F)=c(E)/c(F)$, the quotient of their full Chern classes. Since $c(E)$ and $c(F)$ are graded series 
 starting with one in degree zero, this quotient can be expanded formally and $c_k(E-F)$ is the term of degree $k$. 
 In our setting $E$ is trivial, so
we simply get the inverse of the Chern class of $F$, which is the Segre class of $F^\vee$. We would therefore 
deduce that the ML-degree coincides with 
$$ML^{(0)}_m = \int_{CQ_n}[D(\phi_\cM)]H_{n-1}^{N-m+1}=\int_{CQ_n}s_{m-1}(H_1,2H_1)H_{n-1}^{N-m+1}.$$
We will soon see that this identity is utterly wrong -- but can be corrected. Of course the problems 
come from the exceptional divisors.

\section{First exceptional divisor}
When a matrix $K$ has rank one, $K^2$ is always proportional to $K$. 
Thus $D(\phi_\cM)$ always contains $E_1$, which is certainly not of the expected codimension! 
More precisely $K^2=\mathrm{trace}(K)K$, 
so if we replace $K^2$ by $K^2-\mathrm{tr}(K)K$ in $\phi_\cM$, the 
second component factorizes through $\mathcal{O}(-2H_1+E_1)$, which turns out to be  $\mathcal{O}_{CQ_n}(-H_2)$ (see \cite{mmmsv}): indeed, $2\times 2$ minors span the linear system of quadrics vanishing on rank one matrices. Hence a morphism
$$\phi'_\cM: \mathcal{O}_{CQ_n}(-H_1)\oplus \mathcal{O}_{CQ_n}(-H_2) 
\longrightarrow \mathcal{O}_{CQ_n}\otimes \SS_n/\mathcal{M},$$
which coincides with $\phi_\cM$ outside $E_1$. 

\begin{lemma}
For $\cL$ generic,  $D(\phi'_\cM)\cap q^{-1}\PP(\cL) $ does not meet $E_1$.
\end{lemma}

\begin{proof}
We make a local computation on $E_1$, over the rank one matrix $K=e_1^2$, where $e_1$ is the first vector in the canonical basis of $\CC^n$. In other words, the entries of $K$ are $K_{ij}=\delta_{i1}\delta_{j1}$. Locally around $[K]$ in $\PP(\SS_n)$
we get local coordinates by restricting to matrices $X$ with $X_{11}=1$, and  $D_1$ is cut out by the local equations 
$X_{ij}=X_{1i}X_{1j}$, for $i,j\ge 2$. We can therefore describe the blowup along $D_1$ locally around $[K]$ as the set 
of pairs $(X,[Y])$ with $Y=(Y_{ij})_{i,j\ge 2}\in\SS_{n-1}$ a nonzero symmetric matrix such that 
$X_{ij}-X_{1i}X_{1j}=tY_{ij}$ for all $i,j$,  for some scalar $t$. A straightworward computation then shows that 
$$X^2-\mathrm{tr}(X)X=t\begin{pmatrix} -\mathrm{tr}(Y) & 0\\ 0& Y\end{pmatrix} +O(t^2). $$
Since $(t=0)$ is the equation of the exceptional divisor $E_1$ on the blowup, the fact that $t$ factorizes 
confirms that the morphism can be extended from  $\mathcal{O}(-2H_1)$ to $\mathcal{O}(-2H_1+E_1)$, which 
amounts precisely to dividing the above expression by $t$. Then letting $t=0$, we get the morphism $\phi'_\cM$ 
restricted to $E_1$, and we see that its image at $([K],[Y])$ is the pencil $\langle K,Y\rangle$ of matrices generated 
by $K$ and $Y$ (the latter being considered as a matrix in $\SS_n$ by letting $Y_{1i}=0$ for all $i$). 

So our claim amounts to saying that for $\cL$ general there is no pair of non zero matrices 
$(K,Y)\in D_1\times D_{n-1}$, with $KY=0$, such that $Y\in \cL$ and $\langle K,Y\rangle \cap\cL^\perp\ne 0$. 
In order to prove this, it is enough to check that the set $S$ of triples $([K],[Y],\cL)$ verifying
the previous conditions has dimension smaller than $m(N+1-m)$: since this number is the dimension of the Grassmannian
$G$ of $m$-dimensional subspaces of $\SS_n$, the image of the projection of $S$ to $G$ will have to be a proper subset 
of $G$, proving the claim.

In order to estimate the dimension of $S$, we will of course project it to $F_2$, whose 
dimension is $N-1$. So we fix $([K],[Y])$ and we ask  $\cL^\perp$ to  be contained in $Y^\perp$ and to meet 
$\langle K,Y\rangle$ non trivially. 

There are two cases. If $\langle K,Y\rangle\cap Y^\perp$ is a line $D$, we need $D\subset \cL^\perp\subset Y^\perp$, 
so that $\cL$ belongs to a Grassmannian of dimension $(m-1)(N-m)$. Adding the parameters for $([K],[Y])$, we get 
$(m-1)(N-m)+(N-1)=m(N+1-m)-1=\dim (G)-1$, as required. If $\langle K,Y\rangle\subset Y^\perp$, then $\cL^\perp$ 
has to belong to the Grassmannian of dimension $N+1-m$ subspaces of $Y^\perp$, which has dimension $(m-1)(N+1-m)$, 
and meet a plane non trivially, which is a codimension $m-2$ condition. Adding at most $N-2$ parameters for $([K],[Y])$,
since they are not generic in $F_2$, 
we get a total of  $(m-1)(N+1-m)-(m-2)+(N-2)=m(N+1-m)-1=\dim(G)-1$, as required again. 
\end{proof}

\smallskip
Applying the Thom-Porteous formula as above, we would get the refined expectation that the ML-degree  
should coincide with
$$ML^{(1)}_m= \int_{CQ_n}s_{m-1}(H_1,H_2)H_{n-1}^{N-m+1}.$$
This will be true for $m\le 3$,  but not for $m\ge 4$, because the next exceptional divisors also needs 
to be taken into account.

\section{Second exceptional divisor}
Suppose $X=a^2+b^2$ has rank two, so that $a$ and $b$ are independent vectors. To compose symmetric matrices 
we use the standard quadratic  form $q$ on $\CC^n$, in which terms we get 
$$X^2 = q(a)a^2+q(b)b^2+2q(a,b)ab.$$
So $X^2$ is proportional to $X$ iff $q(a)=q(b)$ and $q(a,b)=0$.
If teh restriction of $q$ to the plane $U=\langle a,b\rangle $
is non degenerate, this exactly means that $X$ is (up to scalars) 
the restriction to $U$ of the dual quadratic form. 

\smallskip
Now if we cut out $E_2 $with  $q^{-1}\PP(\cL)$, we get pairs of matrices  
$([X],[Y])\in D_2\times D_{n-2}$ with $[Y]\in\cL$. For $\cL$ general
of dimension $m\le 3$, the intersection is therefore empty since 
$\PP(\cL)\cap D_{n-2}=\emptyset$. For $m=4$, $\PP(\cL)\cap D_{n-2}$ is a collection of $\delta_n$ smooth points, where $\delta_n=\binom{n+1}{3}$ is the degree of $D_{n-2}$. Moreover, for each of these points, by the generality assumption the matrix $Y$ has rank exactly $n-2$, and the quadratic form $q$ is non degenerate on its kernel; so $[X]$ is uniquely determined. 
%
%, which happens in codimension three, along 
%$$\{q_U\}\times_{G(2,n)} \PP(S^2\cQ^*)\subset %F_2=\PP(S^2\cU)\times_{G(2,n)} \PP(S^2\cQ^*).$$
%This locus is defined on $F_2$ by two linear conditions on the first %factor $\PP(S^2\cU)\simeq\PP^2$, 
%so the fundamental class of its preimage in $CQ_n$ is just %$E_2H_1^2$. 

\medskip
We now have enough information to prove our main result.

\section{Proof of the Theorem} Let us summarize our discussion. 
We have seen that the ML-degree counts the number of points in a finite intersection 
inside the open subset $CQ_n^0$ of the variety of complete quadrics. Globally over $CQ_n$, we have expressed 
this intersection as that of the degeneracy locus $D(\phi'_\cM)$ with the pre-image $q^{-1}\PP(\cL)$
of a linear space. If this intersection is finite, then we can compute its degree as an intersection 
number in the variety of complete quadrics, between the class of $D(\phi'_\cM)$ in the Chow ring, and the class
of $q^{-1}\PP(\cL)$. On the one hand, if $D(\phi'_\cM)$ has the expected dimension, then its class is 
given by the Thom-Porteous formula, which yields a Segre class $s_{m-1}(H_1,H_2)$. On the other hand, the class 
of  $q^{-1}\PP(\cL)$ is simply a power of the pull-back by $q$ of the hyperplane class. So the relevant 
intersection  number on $CQ_n$ can be computed. 
But we have to be careful about the intersection points that 
may belong to the exceptional divisors, which should not be counted in the ML-degree. This yields a 
relation 
$$\int_{CQ_n}s_{m-1}(H_1,H_2)H_{n-1}^{N-m+1}=ML_m+\Delta_m^{(1)}+\cdots +\Delta_m^{(n-1)},$$
where $\Delta_m^{(i)}$ is the contribution of $E_i$ to our intersection problem. We have seen that $\Delta_m^{(1)}=0$.
We claim that for dimensional reasons,
$$\Delta_m^{(i)}=0 \quad\mathrm{for}\quad m\le  \binom{i+1}{2}.$$
Indeed, $\PP(\cL)$ is a generic linear subspace of dimension $m-1$ in $\PP(\SS_n)$, so it does not meet
the degeneracy locus $D_{n-i}$, when $m-1$ is smaller than the codimension $\binom{i+1}{2}$ of the latter. Since 
$D_{n-i}=q(E_i)$, this implies our claim.

\medskip\noindent $\mathbf{m=2}$. By the previous claims the exceptional divisors do not contribute to our intersection number, and we directly get that 
$$ML_2=\int_{CQ_n}(H_1+H_2)H_{n-1}^{N-1}=(n-1)+(n-2)=2n-3.$$
Indeed, $H_{n-1}^{N-1}$ is represented by a projective line of
matrices, and $H_1$ (resp. $H_2$) by a linear relation between 
the maximal (resp. submaximal)  minors of these matrices, which of course have degree $n-1$ (resp. $n-2$). 

\medskip\noindent $\mathbf{m=3}$. Here again the exceptional divisors do not contribute, hence
$$ML_3=\int_{CQ_n}s_{2}(H_1,H_2)H_2^{N-2}=
\int_{CQ_n}(H_1^2+H_1H_2+H_2^2)H_{n-1}^{N-2}.$$
By the previous interpretation this yields the expected result:
$$ML_3=(n-1)^2+(n-1)(n-2)+(n-2)^2=3n^2-9n+ 7.$$  

\medskip\noindent $\mathbf{m=4}$. Here $E_2$ has to be taken into account, and we have %seen how to compute $\Delta_3^{(2)}$. We get 
$$ML_4=\int_{CQ_n}s_{3}(H_1,H_2)H_{n-1}^{N-3}-\Delta_4^{(2)}.$$
%\int_{CQ_n}E_2H_1^2H_{n-1}^{N-3}.$$
In order to compute these numbers we argue as follows.
%(the 
%computations can be double-checked using the results of \cite{mmmsv}).
As before $H_{n-1}^{N-3}$ is represented by a generic $\PP^3$ 
of symmetric matrices,
and its intersection with $H_1^aH_2^{3-a}$ is represented by a complete intersection of $a$ hypersurfaces of degree $n-1$ and 
$3-a$ hypersurfaces of degree $n-2$. By Bertini these hypersurfaces intersect in general transversely outside the base loci of the corresponding linear systems,  which we can avoid since the base locus of $H_2$ has codimension six; with a 
caveat  when $a=3$, 
in which case we only have $H_1$ whose base locus has only 
codimension 
three and cannot be avoided, and then we need  
to substract the degree $\delta_n$ of the variety of corank two matrices (see \cite[section 2.2]{uhler}).  
This yields
$$\int_{CQ_n}H_2^3H_{n-1}^{N-3}=(n-2)^3, \qquad \int_{CQ_n}H_1H_2^2H_{n-1}^{N-3}=(n-1)(n-2)^2,$$
$$\int_{CQ_n}H_1^2H_2H_{n-1}^{N-3}=(n-1)^2(n-2), \qquad \int_{CQ_n}H_1^3H_{n-1}^{N-3}=(n-1)^3-\delta_n.$$
%Moreover, modulo the higher exceptional divisors (which will not %contribute), we have the relations $nH_1=(n-1)E_1+(n-2)E_2$ and %$E_2=2H_1-E_1$, hence $E_2=\frac{n-1}{n-2}H_2-H_1$, from which we %deduce that 
%$$\int_{CQ_n}E_2H_1^2H_{n-1}^{N-3}=\frac{n-1}{n-2}\int_{CQ_n}H_1^2H_2H_{n-1}^{N-3}-\int_{CQ_n}H_1^3H_{n-1}^{N-3}=\delta_n.$$
Finally, we have seen in the previous section that 
$$\Delta_4^{(2)}=\delta_n=\binom{n+1}{3}.$$
Putting all this together we finally conclude that 
$$ML_4 =(n-2)^3+(n-1)(n-2)^2+(n-1)^2(n-2)+(n-1)^3-2%\delta_n,$$
\binom{n+1}{3},$$
which is exactly the conjectured formula.

\section{Some questions} 
How could we go beyond the results of this note? 
\begin{enumerate}
\item For $m\ge 5$ our degeneracy locus always contains the $\cC_2\subset E_2$, whose codimension (three)
is smaller than the expected codimension $m-1$. This is a serious problem in order to compute $ML_m$ 
from intersection theory on complete quadrics. 
Excess intersection theory deals with this sort of situations and might allow to overcome the problem, at least 
up to $m=6$, after which $E_3$ will also enter the show. 

More directly, one could try to 
blow-up $CQ_n$ along $\cC_2$ and try to define a new morphism 
$\phi''_\cM$ on the blow-up, whose degeneracy locus could hopefully be of the correct codimension. 
\item  By the same
argument as for$E_2$, each exceptional divisor $E_s$  contains a component 
%$$\cC_s=\{q_U\}\times_{G(s,n)} \PP(S^2Q^*)\subset E_s=\PP(S^2T)\times_{G(s,n)} \PP(S^2Q^*)$$
$\cC_s$ of the degeneracy locus, of codimension $\binom{s+1}{2}-1$. Moreover this component should contribute 
for $m>\binom{s+1}{2}$. Is there a simple formal argument to prove that this contribution is polynomial 
in $n$? 
\item More generally, is there any formal reason to expect that, like for generic linear concentration models, 
the ML-degrees of generic linear covariance models should be polynomial in $n$? This would again be a very 
remarkable phenomenon, but our concrete evidence for that is still rather limited. 
\end{enumerate}

\bibliographystyle{alpha}

\bigskip
\noindent
\textsc{Institut de Mathématiques de Toulouse, UMR 5219,  Universit\'e Paul Sabatier, F-31062 Toulouse Cedex 9, France}

\noindent
\textit{Email address}: \texttt{manivel@math.cnrs.fr}

\end{document}